\documentclass[12pt]{amsart}
\usepackage{amstext,amsfonts,amssymb,amscd,amsbsy,amsmath,verbatim,diagrams}
\usepackage{ifthen}
\usepackage{color,tikz}
\usepackage{amsthm}
\usepackage{latexsym}
\usepackage[all]{xy}
\usepackage{enumerate}

\newtheorem{lemma}{Lemma}[section]

\newtheorem{propo}[lemma]{Proposition}
\newtheorem{prop}[lemma]{Proposition}

\newtheorem{conj}[lemma]{Conjecture}

\newtheorem{claim*}{Claim}
\newtheorem{thm}[lemma]{Theorem}
\newtheorem{defn}[lemma]{Definition}
\newtheorem{example}[lemma]{Example}

\theoremstyle{remark}
\newtheorem{remark}[lemma]{Remark}
\newtheorem{rmk}[lemma]{Remark}

\newcommand{\coker}{\operatorname{coker}}
\newcommand{\reg}{\operatorname{reg}}

\newcommand{\im}{\operatorname{im}}

\newcommand{\rank}{\operatorname{rank}}
\newcommand{\codim}{\operatorname{codim}}

\newcommand{\defi}[1]{{\upshape\sffamily #1}}

\newcommand{\BN}{B_{\mathbb N}}
\newcommand{\Bmod}{B_{\operatorname{mod}}}
\newcommand{\BQ}{B_{\mathbb Q}}
\newcommand{\Buchs}{\operatorname{Buchs}_{\bullet}}
\newcommand{\pdim}{\operatorname{pdim}}
\newcommand{\favorite}{\begin{pmatrix}  1&2&-&- \\ -&-&2&1\end{pmatrix}}

\begin{document}
\title{The Semigroup of Betti Diagrams}
\author{Daniel Erman}
\thanks{The author was partially supported by an NDSEG fellowship.}
\address{Department of Mathematics, University of California,
        Berkeley, CA 94720-3840, USA}
\email{derman@math.berkeley.edu}
\urladdr{http://math.berkeley.edu/\~{}derman}

\begin{abstract}
The recent proof of the Boij-S\"oderberg conjectures reveals new structure about Betti diagrams of modules, giving a complete description of the cone of Betti diagrams.  We begin to expand on this new structure by investigating the semigroup of Betti diagrams.  We prove that this semigroup is finitely generated, and we answer several other fundamental questions about this semigroup.
\end{abstract}
\maketitle

\section{Introduction}\label{sec:intro}
Recent work of a number of authors (\cite{boij-sod}, \cite{efw}, \cite{eis-schrey}, \cite{boij-sod2}) completely characterizes the structure of Betti diagrams of graded modules, but only if we allow one to take arbitrary rational multiples of the diagrams.  This Boij-S\"oderberg theory shows that the rational cone of Betti diagrams is a simplicial fan whose rays and facet equations have a remarkably simple description.\footnote{See \cite{boij-sod} for the original conjecture, \cite{eis-schrey} for the Cohen-Macaulay case, and \cite{boij-sod2} for the general case.  The introduction of \cite{eis-schrey} includes a particularly clear exposition of the main results.}

In this note, we consider the integral structure of Betti diagrams from the perspective of Boij-S\"oderberg theory, and we begin to survey this new landscape.  In particular, we replace the cone by the semigroup of Betti diagrams (see Definition \ref{defn:semigroup} below) and answer several fundamental questions about the structure of this semigroup.

We first use the results of Boij-S\"oderberg theory to draw conclusions about the semigroup of Betti diagrams.  This comparison leads to Theorem \ref{thm:fingen}, that the semigroup of Betti diagrams is finitely generated.

We then seek conditions which prevent a diagram from being the Betti diagram of an actual module.  Using these conditions, we build families of diagrams which are {\em not} the Betti diagram of any module.  For instance, consider the family:
 $$E_\alpha:= \begin{pmatrix}2+\alpha&3&2&- \\ -&5+6\alpha& 7+8\alpha&3+3\alpha\end{pmatrix},\ \ \  \alpha\in \mathbb N$$
We will use the theory of Buchsbaum-Eisenbud multiplier ideals to conclude that no member of this family can be the Betti diagram of a module.  Yet each $E_\alpha$ belongs to the cone of Betti diagrams, and in fact, if we were to multiply any diagram $E_\alpha$ by $3$, then the result {\em would} equal the Betti diagram of a module.

We produce further examples of obstructed diagrams by using properties of the Buchsbaum-Rim complex.  Based on our examples, we establish several negative results about the semigroup of Betti diagrams.  These negative results are summarized in Theorem \ref{thm:negative}.

To state our results more precisely, we introduce notation.  Let $S$ be the polynomial ring $S=k[x_1, \dots, x_n]$ where $k$ is any field.  If $M$ is any finitely generated graded $S$-module then we can take a minimal free resolution:
$$0\to F_p \to \dots \to F_1\to F_0\to M \to 0$$
with $F_i=\oplus_j S(-j)^{\beta_{i,j}(M)}$.  We write $\beta(M)$ for the Betti diagram of $M$, and we think of $\beta(M)$ as an element of the vector space $\oplus_{j=-\infty}^\infty \oplus_{i=0}^p \mathbb Q$ with coordinates $\beta_{i,j}(M)$.  The set of graded $S$-modules is a semigroup under the operation of direct sum, and the vector space is a semigroup under addition.  By observing that $\beta(M\oplus M')=\beta(M)+\beta(M')$, we can think of $\beta$ as a map of semigroups:
$$\{\text{ fin. gen'd graded } S-\text{modules} \} \rTo^{\beta} \bigoplus_{j=-\infty}^\infty \bigoplus_{i=0}^p \mathbb Q$$
The image of this map is thus a semigroup.  Furthermore, if we restrict $\beta$ to any subsemigroup of $S$-modules, then the image of the restricted map is also a semigroup.

A \defi{degree sequence} will mean an integral sequence $d=(d_0, \dots, d_p)\in \mathbb N^{p+1}$ where $d_i<d_{i+1}$.  If there exists a Cohen-Macaulay module $M$ of codimension $p$ with all Betti numbers equal to zero except for $\beta_{i,d_i}(M)$, then we say that $\beta(M)$ is a \defi{pure diagram of type} $d$.  It was first shown in \cite{herz-kuhl} that any two pure diagrams of type $d$ would be scalar multiples of one another.  The existence of modules whose Betti diagrams are pure diagrams of type $d$ was conjectured by \cite{boij-sod} and proven by \cite{efw} in characteristic $0$ and by \cite{eis-schrey} in arbitrary characteristic.  These pure diagrams play a central role in the Boij-S\"oderberg theorems.

Fix two degree sequences $\underline{d}$ and $\overline{d}$ of length $p$ and such that $\underline{d}_i\leq \overline{d}_i$ for all $i$.  Consider the semigroup $\mathcal Z$ of graded modules $M$ which satisfy the properties:
\begin{itemize}
    \item  $M$ has projective dimension $\leq p$
    \item  The Betti number $\beta_{i,j}(M)$ is nonzero only if $i\leq p$ and $\underline{d_i}\leq j \leq \overline{d}_i$.
\end{itemize}
Our choice of $\mathcal Z$ is meant to match the simplicial structure of the cone of Betti diagrams.  We may now define our main objects of study.
\begin{defn}\label{defn:semigroup}
The \defi{semigroup of Betti diagrams} $\Bmod$ is defined as:
$$\Bmod=\Bmod(\underline{d}, \overline{d}):=\im \beta|_{\mathcal Z}$$
\end{defn}

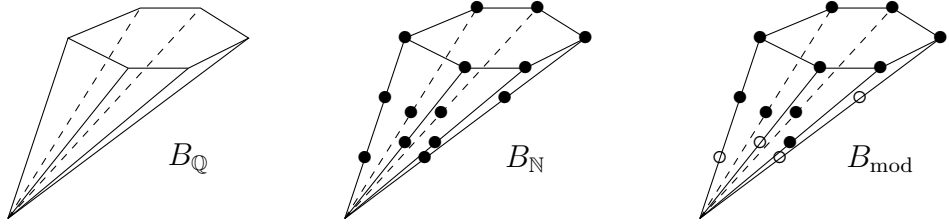
\begin{figure}
\begin{tikzpicture}[scale=0.8]
\draw[-](0,0)--(1,3);
\draw[-](0,0)--(2,2.5);
\draw[-](0,0)--(3,2.5);
\draw[-](0,0)--(4,3);
\draw[dashed,-](0,0)--(3.2,3.5);
\draw[dashed,-](0,0)--(2.2,3.5);
\draw[-](1,3)--(2,2.5)--(3,2.5)--(4,3)--(3.2,3.5)--(2.2,3.5)--cycle;
\draw (3,1) node {$B_{\mathbb Q}$};
\end{tikzpicture}
\hspace{1cm}
\begin{tikzpicture}[scale=0.8]
\draw[-](0,0)--(1,3);
\draw (.33,1) node {$\bullet$};
\draw (.67,2) node {$\bullet$};
\draw (1,3) node {$\bullet$};
\draw[-](0,0)--(2,2.5);
\draw (1,1.25) node {$\bullet$};
\draw (2,2.5) node {$\bullet$};
\draw[-](0,0)--(3,2.5);
\draw (1.5,1.25) node {$\bullet$};
\draw (3,2.5) node {$\bullet$};
\draw[-](0,0)--(4,3);
\draw (1.33,1) node {$\bullet$};
\draw (2.66,2) node {$\bullet$};
\draw (4,3) node {$\bullet$};
\draw[dashed,-](0,0)--(3.2,3.5);
\draw (1.6,1.75) node {$\bullet$};
\draw (3.2,3.5) node {$\bullet$};
\draw[dashed,-](0,0)--(2.2,3.5);
\draw (1.1,1.75) node {$\bullet$};
\draw (2.2,3.5) node {$\bullet$};
\draw[-](1,3)--(2,2.5)--(3,2.5)--(4,3)--(3.2,3.5)--(2.2,3.5)--cycle;
\draw (3,1) node {$\BN$};
\end{tikzpicture}
\hspace{1cm}
\begin{tikzpicture}[scale=0.8]
\draw[-](0,0)--(1,3);
\draw (.33,1) node {$\circ$};
\draw (.67,2) node {$\bullet$};
\draw (1,3) node {$\bullet$};
\draw[-](0,0)--(2,2.5);
\draw (1,1.25) node {$\circ$};
\draw (2,2.5) node {$\bullet$};
\draw[-](0,0)--(3,2.5);
\draw (1.5,1.25) node {$\bullet$};
\draw (3,2.5) node {$\bullet$};
\draw[-](0,0)--(4,3);
\draw (1.33,1) node {$\circ$};
\draw (2.66,2) node {$\circ$};
\draw (4,3) node {$\bullet$};
\draw[dashed,-](0,0)--(3.2,3.5);
\draw (1.6,1.75) node {$\bullet$};
\draw (3.2,3.5) node {$\bullet$};
\draw[dashed,-](0,0)--(2.2,3.5);
\draw (1.1,1.75) node {$\bullet$};
\draw (2.2,3.5) node {$\bullet$};
\draw[-](1,3)--(2,2.5)--(3,2.5)--(4,3)--(3.2,3.5)--(2.2,3.5)--cycle;
\draw (3,1) node {$\Bmod$};
\end{tikzpicture}
\label{fig:semigroups}
\caption{The cone of Betti diagrams $B_{\mathbb Q}$ is a simplicial fan which is described explicitly in \cite{eis-schrey} and \cite{boij-sod2}.  This explicit description can be used to understand the integral structure of the semigroup of virtual Betti diagrams $\BN$.  The semigroup of Betti diagrams $\Bmod$ is more mysterious.}
\end{figure}
In order to study the semigroup of Betti diagrams, it will be useful to consider two related objects:
\begin{defn}
The \defi{cone of Betti diagrams} $B_{\mathbb Q}$ is the positive rational cone over the semigroup of Betti diagrams.
The \defi{semigroup of virtual Betti diagrams} $\BN$ is the semigroup of lattice points in $B_{\mathbb Q}$.
\end{defn}
One could define a cone of Betti diagrams without restricting which Betti numbers can be nonzero.  This is the choice that \cite{eis-schrey} make, and our cone of Betti diagrams equals this big cone of \cite{eis-schrey} restricted to an interval.  We choose to work with a finite dimensional cone in order to discuss the finiteness properties of $\Bmod$.

A naive hope would be that the semigroups $\BN$ and $\Bmod$ are equal.  But a quick search yields virtual Betti diagrams which cannot equal the Betti diagram of module.  Take for example the following pure diagram of type $(0,1,3,4)$:
\begin{equation}\label{eqn:oldfavorite}
D_1:=\pi_{(0,1,3,4)}=\begin{pmatrix}1&2&-&-\\ -&-&2&1  \end{pmatrix}
\end{equation}
This diagram belongs to the semigroup of virtual Betti diagrams.  However, $D_1$ cannot equal the Betti diagram of an actual module as the two first syzygies would satsify a linear Koszul relation which does not appear in the diagram $D_1$.

It is thus natural to compare $\Bmod$ and $\BN$, and we will consider the following questions about the semigroup of Betti diagrams:

\begin{enumerate}
    \item[(Q1)]\label{question:fingen}  Is $\Bmod$ finitely generated?
    \item[(Q2)]   Does $\Bmod=\BN$ in some special cases?
    \item[(Q3)]   Is $\Bmod$ a saturated semigroup?
    \item[(Q4)]   Is $\BN\setminus \Bmod$ a finite set?
    \item[(Q5)]   On a single ray, can we have consecutive points of $\BN$ which fail to belong to $\Bmod$?  Nonconsecutive points?
\end{enumerate}

In Section \ref{sec:fingen}, we answer (Q\ref{question:fingen}) affirmatively:
\begin{thm}\label{thm:fingen}
The semigroup of Betti diagrams $\Bmod$ is finitely generated.
\end{thm}

Sections \ref{sec:brobs} and \ref{sec:linstrand} of this paper develop obstructions which prevent a virtual Betti diagram from being the diagram of some module. These obstructions are our tools for answering the other questions above.  In section \ref{sec:specialcases}, we consider (Q2), and prove the following:
\begin{prop}\label{prop:specialcases}
$\BN=\Bmod$ for projective dimension $1$ and for projective dimension $2$ level modules.
\end{prop}
Our proof of Proposition \ref{prop:specialcases} rests heavily on \cite{sod-codim2}, which shows the existence of level modules of embedding dimension $2$ and with a given Hilbert function by constructing these modules as quotients of monomial ideals.

In \cite{erman-thesis} we verify that, in a certain sense, projective dimension $2$ diagrams generated in a single degree are ``unobstructed.''  This leads us to conjecture:
\begin{conj}\label{conj:codim2}
$\BN=\Bmod$ for projective dimension $2$ diagrams.
\end{conj}

In the final section, we will consider questions (Q3-Q5).  Here we show that the semigroup of Betti diagrams can have rather complicated behavior (see also Figure \ref{fig:rays}):
\begin{thm}\label{thm:negative}
Each of the following occurs in the semigroup of Betti diagrams:
\begin{enumerate}
    \item\label{thm:negative:sat}   $\Bmod$ is not necessarily a saturated semigroup.
    \item\label{thm:negative:infty}   The set $|\BN\setminus \Bmod|$ is not necessarily finite.
    \item\label{thm:negative:consec}   There exist rays of $\BN$ which are missing at least $(\dim S-2)$ consecutive lattice points.
    \item\label{thm:negative:nonconsec}   There exist rays of $\BN$ where the points of $\Bmod$ are nonconsecutive lattice points.
\end{enumerate}
\end{thm}

\begin{figure}\label{fig:rays}
\begin{tikzpicture}[scale=1.2]
	\draw[-,thick] (0,0)node{$\bullet$} --  (1.5,.75);
	\draw[dotted,-,thick](1.5,.75) -- (2,1);
	\draw (0,0)node[below right] {$0$};
	\draw[->,thick] (2,1)--  (3,1.5);
	\draw (.5,.25) node{$\circ$};
	\draw (1,.5) node{$\bullet$};
	\draw (1.5,.75) node{$\bullet$};
	\draw (2,1) node{$\bullet$};
	\draw (2.5,1.25) node{$\bullet$};
	\draw (1.3,-.5) node {{\bf Nonsaturated}};
\end{tikzpicture}
\hspace{.5cm}
\begin{tikzpicture}[scale=1.2]
	\draw[-,thick] (0,0)node{$\bullet$} --  (1.5,.75);
	\draw[dotted,-,thick](1.5,.75) -- (2,1);
	\draw (0,0)node[below right] {$0$};
	\draw[->,thick] (2,1)--  (3,1.5);
	\draw (.5,.25) node{$\circ$};
	\draw (1,.5) node{$\circ$};
	\draw (1.5,.75) node{$\circ$};
	\draw (2,1) node{$\circ$};
	\draw (2.5,1.25) node{$\bullet$};
	\draw (1.3,-.5) node {{\bf Missing Consecutive Points}};
\end{tikzpicture}
\hspace{.5cm}
\begin{tikzpicture}[scale=1.2]
	\draw[->,thick] (0,0)node{$\bullet$}  --  (2.5,1.25);
	\draw (0,0)node[below right] {$0$};
	\draw (.5,.25) node{$\circ$};
	\draw (1,.5) node{$\bullet$};
	\draw (1.5,.75) node{$\circ$};
	\draw (2,1) node{$\bullet$};
	\draw (1.3,-.5) node {{\bf Nonconsecutive Points}};
\end{tikzpicture}
\caption{There exist rays which exhibit each of the above behavior.}
\end{figure}
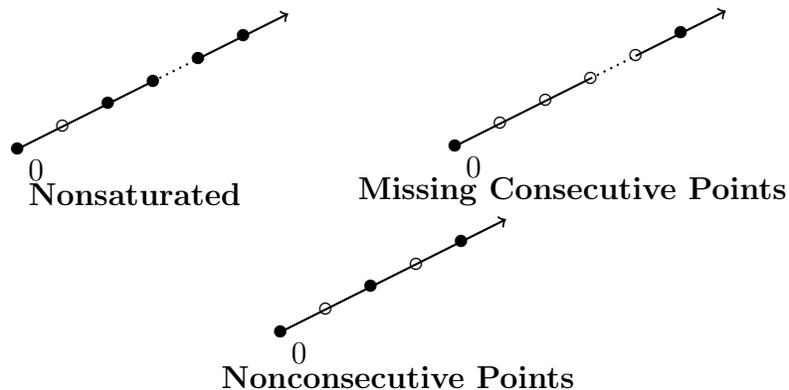

\begin{remark}
Almost nothing in this paper would be changed if we swapped the semigroup $\mathcal Z$ for some subsemigroup of $\mathcal Z$ which respects the simplicial structure of $\BQ$.  For instance, we could consider the subsemigroup of Cohen-Macaulay modules of codimension $e$.  The analogous statements of Theorems \ref{thm:fingen} and \ref{thm:negative} and Proposition \ref{prop:specialcases} all remain true in the Cohen-Macaulay case; one can even use the same proofs.
\end{remark}

This paper is organized as follows.  In Section \ref{sec:fingen}, we prove that the semigroup of Betti diagrams is finitely generated.  Sections \ref{sec:brobs} and \ref{sec:linstrand} introduce obstructions for a virtual Betti diagram to be the Betti diagram of some module. The obstructions in Section \ref{sec:brobs} are based on properties of the Buchsbaum-Rim complex, and the obstruction in Section \ref{sec:linstrand} focuses on the linear strand of a resolution and is based on the properties of Buchsbaum-Eisenbud multiplier ideals.  In Section \ref{sec:specialcases}, we consider the semigroup of Betti diagrams for small projective dimension, and we prove Proposition \ref{prop:specialcases}.  In Section \ref{sec:structure} we prove Theorem \ref{thm:negative} by constructing explicit examples based on our obstructions.  Finally, Section \ref{sec:questions} offers some open questions.

\section*{Acknowledgments}
I thank David Eisenbud for his suggestions and encouragement throughout my work on this project.  I also  thank Marc Chardin, Frank Schreyer, and the participants in the 2008 Minimal Resolutions conference at Cornell University for many useful discussions.  Finally, I thank the makers of Macaulay2 \cite{macaulay2} and Singular \cite{singular}.

\section{Finite Generation of the Semigroup of Betti Diagrams}\label{sec:fingen}
We fix a pair of degree sequences $\overline{d}, \underline{d}\in \mathbb N^{p+1}$ and work with the corresponding  semigroup of Betti diagrams $\Bmod$.  Our proof of the finite generation of the semigroup of Betti diagrams uses the structure of the cone of Betti diagrams, so we begin by reviewing the relevant results.  This structure was first proven in \cite{eis-schrey} for the Cohen-Macaulay case; the general case is similar, and was worked out in \cite{boij-sod2}.

If $d$ is any degree sequence then we set $\pi_d$ to be the first lattice point on the ray corresponding to $d$. As illustrated in Figure 3,
the cone $\BQ$ is a rational simplicial fan whose defining rays correspond to rays of pure diagrams.  To describe the simplicial structure, we recall the following partial ordering on degree sequences, introduced in \cite{boij-sod2}:
\begin{defn}
Let $d \in \mathbb N^{t+1}$ and $d'\in \mathbb N^{u+1}$.  Then $d \leq d'$ if $t\geq u$ and $d_i\leq d'_i$ for all $i\leq u$.
\end{defn}

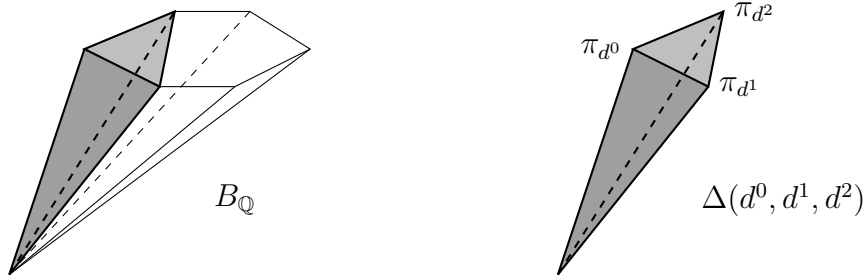
\begin{figure}\label{fig:simpl}
\begin{tikzpicture}[scale=1.0]
\fill[fill=gray,semitransparent](0,0)--(1,3)--(2.2,3.5);
\fill[fill=gray,semitransparent](0,0)--(1,3)--(2,2.5);
\fill[fill=gray,semitransparent](0,0)--(2,2.5)--(2.2,3.5);
\draw[-,thick](0,0)--(1,3);
\draw[-,thick](0,0)--(2,2.5);
\draw[-](0,0)--(3,2.5);
\draw[-](0,0)--(4,3);
\draw[dashed,-](0,0)--(3.2,3.5);
\draw[dashed,-,thick](0,0)--(2.2,3.5);
\draw[-,thick](1,3)--(2.2,3.5)--(2,2.5)--cycle;
\draw[-](1,3)--(2,2.5)--(3,2.5)--(4,3)--(3.2,3.5)--(2.2,3.5)--cycle;
\draw (3,1) node {$B_{\mathbb Q}$};
\end{tikzpicture}
\hspace{3cm}
\begin{tikzpicture}[scale=1.0]
\fill[fill=gray,semitransparent](0,0)--(1,3)--(2.2,3.5);
\fill[fill=gray,semitransparent](0,0)--(1,3)--(2,2.5);
\fill[fill=gray,semitransparent](0,0)--(2,2.5)--(2.2,3.5);
\draw[-,thick](0,0)--(1,3);
\draw  (1,3) node [left] {$\pi_{d^0}$};
\draw[-,thick](0,0)--(2,2.5);
\draw  (2,2.5) node [right] {$\pi_{d^1}$};
\draw[dashed,-,thick](0,0)--(2.2,3.5);
\draw  (2.2,3.5) node [right] {$\pi_{d^2}$};
\draw[-,thick](1,3)--(2.2,3.5)--(2,2.5)--cycle;
\draw (3,1) node {$\Delta(d^0, d^1, d^2)$};
\end{tikzpicture}
\caption{The cone $B_{\mathbb Q}$ is a simplicial fan.  The simplex corresponding to a maximal sequence $d^0,d^1,d^2$ is highlighted in gray.  The extremal rays of a simplex correspond to pure diagrams.}
\end{figure}
The simplices of the fan $\BQ$ correspond to maximal chains of degree sequences:
$$d^0 < d^1 < \dots < d^{s-1} < d^s$$
where if $d^j\in \mathbb N^{t+1}$ then $\underline{d}_i\leq d^j_i \leq \overline{d}_i$ for all $i\leq t$.  There are thus $s+1$ positions which may be nonzero for a Betti diagram in $\Bmod$ (see Example 1 of \cite{boij-sod2}).  In particular, $s+1=\sum_{i=0}^p \overline{d}_i-\underline{d}_i+1$.

Before proving Theorem \ref{thm:fingen}, we first prove a simpler analog for the semigroup of virtual Betti diagrams $\BN$.

\begin{lemma}
The semigroup $\BN$ is finitely generated.  There exists $m$ such every virtual Betti diagram can be written as a $\frac{1}{m}\mathbb N$-combination of pure diagrams.
\end{lemma}
\begin{proof}
Since $\BN$ consists of the lattice points of the simplicial fan $\BQ$, it is sufficient to prove this lemma after restricting to a single simplex $\Delta$.  Let $\pi_{d^0}, \dots, \pi_{d^s}$ be the pure diagrams defining $\Delta$.  Then the semigroup $\BN \cap \Delta$ is generated by pure diagrams spanning $\Delta$ and by the lattice points inside the fundamental parallelepiped of $\Delta$.  This proves the first claim.

For the second claim of the lemma, let $P_1, \dots, P_N$ be the minimal generators of $\BN\cap \Delta$.  Every generator can be written as a positive rational sum:
$$P_i=\sum_{j} \frac{p_{ij}}{q_{ij}} \pi_{d^j}, \quad p_{ij}, q_{ij}\in \mathbb N$$
We set $m_\Delta$ to be the least common multiple of all the $q_{ij}$.  Then we set $m$ to be the least common multiple of the $m_\Delta$ for all $\Delta$.
\end{proof}
We refer to $m_\Delta$ as a \defi{universal denominator for } $\BN \cap \Delta$.  The existence of this universal denominator is central to our proof of the finite generation of $\Bmod$.

\begin{proof}[Proof of Theorem \ref{thm:fingen}]
It is sufficient to prove the theorem for $\Bmod \cap \Delta$ where $\Delta$ is a simplex of $\BQ$.  Let $\pi_{d^0}, \dots, \pi_{d^s}$ be the pure diagrams defining $\Delta$, and let $m_\Delta$ be the universal denominator for $\BN \cap \Delta$.

For $i=0, \dots, s$, let $c_i\in \mathbb N$ be minimal such that $c_i\pi_{d^i}$ belongs to $\Bmod$.  The existence of such a $c_i$ is guaranteed by Theorems 0.1 and 0.2 of \cite{efw} and Theorem 5.1 of \cite{eis-schrey}.  Let $\mathcal S_1$ be the semigroup generated by the pure diagrams $c_i\pi_{d^i}$.   Let $\mathcal S_0$ be the semigroup generated by the pure diagrams $\frac{1}{m_\Delta}\pi_{d^i}$.  Then we have the following inclusions of semigroups:
$$\mathcal S_1 \subseteq \left(\Bmod\cap \Delta \right) \subseteq \left( \BN\cap \Delta\right) \subseteq \mathcal S_0$$
Passing to semigroup rings gives:
$$k[\mathcal S_1] \subseteq k[\Bmod\cap \Delta] \subseteq k[ \BN\cap \Delta] \subseteq k[\mathcal S_0]$$
Observe that $k[\mathcal S_1] $ and $k[\mathcal S_0]$ are both polynomial rings of dimension $s+1$, and that $k[\mathcal S_1]\subseteq k[\mathcal S_0]$ is a finite extension of rings.  This implies that $k[\mathcal S_1]\subseteq k[\Bmod\cap \Delta]$ is also a finite extension, and hence $k[\Bmod\cap \Delta]$ is a finitely generated $k$-algebra.  We conclude that $\Bmod\cap \Delta$ is a finitely generated semigroup.
\end{proof}

\subsection*{Computing Generators of $\BN$}
Minimal generators of $\BN\cap \Delta$ can be computed explicitly as the generators of the $\mathbb N$-solutions to a certain linear $\mathbb Z$-system defined by the $\pi_{d^i}$ and by $m_\Delta$.  For an overview of relevant algorithms, see the introduction of \cite{pison}.  The following example illustrates the method.

Consider $S=k[x, y], \underline{d}=(0,1,4), \overline{d}=(0,3,4)$.  The corresponding cone of Betti diagrams has several simplices and we choose the simplex $\Delta$ spanned by the maximal chain of degree sequences:
$$(0)>(0,3)>(0,3,4)>(0,2,4)>(0,1,4)$$
The corresponding pure diagrams are:
\begin{equation}\label{eqn:comp:pures}
\begin{pmatrix}1&-&-\\ -&-&- \\ -&-&- \end{pmatrix}, \begin{pmatrix}1&-&-\\ -&-&- \\ -&1&- \end{pmatrix}, \begin{pmatrix}1&-&-\\ -&-&- \\ -&4&3 \end{pmatrix}, \begin{pmatrix}1&-&-\\ -&2&- \\ -&-&1 \end{pmatrix}, \begin{pmatrix}3&4&-\\ -&-&- \\ -&-&1 \end{pmatrix}
\end{equation}
First we must compute $m_\Delta$.  To do this, we consider the square matrix $\Phi$ whose columns correspond to the pure above pure diagrams:
\begin{equation}\label{eqn:phi}
\Phi=\left(\begin{smallmatrix}1&1&1&1&3\\0&0&0&0&4\\0&0& 0&2&0\\ 0&1&4&0&0 \\ 0&0&3&1&1  \end{smallmatrix}\right)
\end{equation}
Since the columns of $\Phi$ are $\mathbb Q$-linearly independent, it follows that the cokernel of $\Phi$ will be entirely torsion.  Note that each minimal generator of $\BN\cap \Delta$ is either a pure diagram or corresponds to a unique nonzero torsion element of $\coker(\Phi)$.  The annihilator of $\coker(\Phi)$ is thus the universal denominator for $\Delta$.  A computation in \cite{macaulay2} shows that $m_\Delta=12$ in this case.

We next compute minimal generators of the $\mathbb N$-solutions of the following linear $\mathbb Z$-system:
$$\mathbb Z^{10} \rTo^{\left(\begin{smallmatrix}-12&0&0&0&0&1&1&1&1&3\\
0&-12&0&0&0&0&0&0&0&4\\
0&0&-12&0&0&0&0& 0&2&0\\
0&0&0&-12&0&0&1&4&0&0 \\
0&0&0&0&-12&0&0&3&1&1  \end{smallmatrix}\right) } \mathbb Z^5
$$
The $\mathbb N$-solutions of the above system correspond to elements of $\BN\cap \Delta$ under the correspondence:
$$(b_1, b_2, b_3, b_4, b_5, a_1, a_2, a_3,a_4,a_5) \mapsto \frac{a_1}{12} \pi_{(0)}+\frac{a_2}{12}\pi_{(0,3)}+\frac{a_3}{12}\pi_{(0,3,4)}+\frac{a_4}{12}\pi_{(0,2,4)}+\frac{a_5}{12}\pi_{(0,1,4)}$$
Computation yields that $\BN\cap \Delta$ has $14$ minimal semigroup generators.\footnote{We use Algorithm 2.7.3 of \cite{sturminv} for this computation.  Also, see \cite{pison} for other relevant algorithms.}  These consist of the $5$ pure diagrams from line (\ref{eqn:comp:pures}) plus the following $9$ diagrams:
\begin{eqnarray*}
\begin{pmatrix}1&1&-\\ -&-&- \\ -&1&1 \end{pmatrix}, \begin{pmatrix}2&2&-\\ -&1&- \\ -&-&1 \end{pmatrix}, \begin{pmatrix}1&-&-\\ -&1&- \\ -&2&2 \end{pmatrix},\begin{pmatrix}1&-&-\\ -&-&- \\ -&2&1 \end{pmatrix}, \begin{pmatrix}2&2&-\\ -&-&- \\ -&1&1  \end{pmatrix},\\
 \begin{pmatrix}3&3&-\\ -&-&- \\ -&1&1  \end{pmatrix},
\begin{pmatrix}1&-&-\\ -&-&- \\ -&3&2  \end{pmatrix}, \begin{pmatrix}2&1&-\\ -&1&- \\ -&1&1  \end{pmatrix},   \begin{pmatrix}1&-&-\\ -&1&- \\ -&1&1\end{pmatrix},
\end{eqnarray*}
It is not difficult to verify that each of these generators is the Betti diagram of some module.  Thus in this case we have $\BN\cap \Delta=\Bmod \cap \Delta$.

%

\begin{remark}\label{rmk:bndgens}
We can easily bound the number of generators of $\BN\cap \Delta$ from above.  Let $\Delta$ be a simplex spanned by $d^0, \dots, d^s$.  Let $\Phi$ be the square matrix:
$$\Phi: \mathbb Z^{s+1} \to \bigoplus_{i=0}^n \bigoplus_{j=\underline{d}_i}^{\overline{d}_i} \mathbb Z$$
which sends the $\ell$'th generator to the pure diagram $\pi_{d^\ell}$.  As in line (\ref{eqn:phi}), the cokernel of $\Phi$ will be entirely torsion (this follows from Proposition 1, \cite{boij-sod2}.)  Each minimal generator of $\BN\cap \Delta$ will correspond to either a pure diagram or a unique nonzero element of $\coker(\Phi)$.  Since the order of $\coker(\Phi)$ equals the determinant of $\Phi$, the number of generators of $\BN\cap \Delta$ is bounded above by $\det (\Phi)+s$.

We know of no effective upper bound for the number of generators of $\Bmod\cap \Delta$.
\end{remark}

\begin{remark}
Although the semigroup $\BN$ is saturated, the map $k[\Bmod]\to k[\BN]$ may not be the normalization map.  For instance, if there is a ray $r$ such that $r \cap \Bmod$ only contains every other lattice point, then the saturation of $r \cap \Bmod$ will not equal $r \cap \BN$.  Eisenbud, Fl\o ystad and Weyman conjecture that there are no rays corresponding to pure diagrams which have this property \cite{efw}.
\end{remark}

\section{Buchsbaum-Rim obstructions to existence of Betti diagrams}\label{sec:brobs}
In Proposition  \ref{prop:BRobstructions} we illustrate obstructions which prevent a virtual Betti diagram from being the Betti diagram of an actual module.  To yield information not contained in the main results of \cite{eis-schrey} and \cite{boij-sod2}, these obstructions must be sensitive to scalar multiplication of diagrams.  For simplicity we restrict to the case that $M$ is generated in degree $0$, though all of these obstructions can be extended to the general case.

We say that a diagram $D$ is a \defi{Betti diagram} if $D$ equals the Betti diagram of some module $M$, and we say that $D$ is a \defi{virtual Betti diagram} if $D$ belongs to the semigroup of virtual Betti diagrams $\BN$.  Many properties of modules (e.g. codimension, Hilbert function) can be computed directly from the Betti diagram.   We extend such properties to virtual diagrams in the obvious way.  Proposition \ref{prop:BRobstructions} only involves quantities which can be determined entirely from the Betti diagram; thus we may easily test whether an arbitrary virtual Betti diagram is ``obstructed'' in the sense of this proposition.

\begin{propo}[Buchsbaum-Rim obstructions]\label{prop:BRobstructions}
Let $M$ a graded module of codimension $e\geq 2$ with minimal presentation:
$$\bigoplus_{\ell=1}^b S(-j_\ell)\rTo^{\phi} S^a\rTo M\rTo 0$$
Assume that $j_1\leq j_2\leq \dots \leq j_b$.  Then we have the following obstructions:
\begin{enumerate}
    \item\label{prop:BR:2nd} (Second syzygy obstruction): $$\underline{d_2}(M)\leq \sum_{\ell=1}^{a+1}j_{\ell}$$
    \item\label{prop:BR:codim}  (Codimension obstruction)  $$b=\sum_{j} \beta_{1,j}(M) \geq e+a-1$$  If we have equality, then $\beta(M)$ must equal the Betti diagram of the Buchsbaum-Rim complex of $\phi$.
    \item\label{prop:BR:reg}  (Regularity obstruction in Cohen-Macualay case): If $M$ is Cohen-Macaulay then we also have that
    $$\reg(M)+e=\overline{d_e}(M)\leq  \sum_{\ell=b-e-a+2}^{b}j_\ell$$
\end{enumerate}
These obstructions are independent of one another, and each obstruction occurs for some virtual Betti diagram.
\end{propo}
In addition, note that both the weak and strong versions of the Buchsbaum-Eisenbud-Horrocks rank conjecture about minimal Betti numbers (see \cite{buchs-eis-gor}or \cite{charalamb} for a description) would lead to  similar obstructions.  Since each Buchsbaum-Eisenbud-Horrocks conjecture imposes a condition on each column of the Betti diagram, the corresponding obstruction would greatly strengthen part \eqref{prop:BR:codim} of the above proposition.


\begin{remark}\label{rmk:duality}
For $D$ a diagram, let $D^\vee$ be the diagram obtained by rotating $D$ by $180$ degrees.  When $D$ is the Betti diagram of a Cohen-Macaulay module $M$ of codimension $e$, then $D^\vee$ is the Betti diagram of some twist of $M^\vee:= Ext^{e}_S(M,S)$, which is also a Cohen-Macaulay module of  codimension $e$.  Thus, in the Cohen-Macaulay case, we may apply these obstructions to $D$ or to $D^\vee$.
\end{remark}

Given any map $\widetilde{\phi}$ between free modules $F$ and $G$, we can construct the Buchsbaum-Rim complex on this map, which we denote as $\Buchs (\widetilde{\phi})$.  The Betti table of the complex $\Buchs(\widetilde{\phi})$ will depend only on the Betti numbers of $F$ and $G$, and it can be thought of as an approximation of the Betti diagram of the cokernel of $\widetilde{\phi}$.

As in the statement of Proposition \ref{prop:BRobstructions}, let $M$ be a graded $S$-module of codimension $\geq 2$ with minimal presentation:
$$F_1:=\bigoplus_{\ell=1}^b S(-j_\ell)\rTo^{\phi} S^a\rTo M\rTo 0$$
We will consider free submodules $\widetilde{F_1}\subseteq F_1$, the induced map $\widetilde{\phi}: \widetilde{F_1}\to S^a$, and the Buchsbaum-Rim complex on $\widetilde{\phi}$.  By varying $\widetilde{\phi}$ we will produce the obstructions listed in Proposition \ref{prop:BRobstructions}.

To prove the first obstruction, we introduce some additional notation.  Let the first syzygies of $M$ be $\sigma_1, \dots, \sigma_b$ with degrees $\deg(\sigma_\ell)=j_\ell$. The first stage of the Buchsbaum-Rim complex on $\phi$ is the complex:
$$\bigwedge ^{a+1} F_1 \rTo^{\epsilon} F_1 \to S^a$$
A basis of $\bigwedge^{a+1} F_1$ is given by $e_{I'}$ where $I'$ is a subset $I'\subseteq \{1, \dots, b\}$ with $|I'|=a+1$.  Let $\det(\phi_{I'\setminus \{i\}})$ be the maximal minor corresponding to the columns $I'\setminus \{i\}$. Then the map $\epsilon$ sends $e_{I'}\mapsto \sum_{i\in I'} e_i\det(\phi_{I'\setminus \{i\}})$.  We refer to $\epsilon(e_{I'})$ as a \defi{Buchsbaum-Rim second syzygy}, and we denote it by $\rho_{I'}$.  There are $\binom{b}{a+1}$ Buchsbaum-Rim second syzygies.  It may happen that one of these syzygies specializes to $0$ in the case of $\phi$.  But as we now prove, if $\rho_{I'}$ specializes to $0$ then we can find another related syzygy in lower degree.

\begin{lemma}\label{lem:brcol1}
Let $I'=\{i_1, \dots, i_{a+1} \}\subseteq \{1, \dots, b\}$, and assume that $\rho_{I'}$ is a trivial second syzygy.  Then $M$ has a second syzygy of degree strictly less than $\sum_{i\in I'} j_{i}$ and supported on a subset of the columns corresponding to $I'$.
\end{lemma}
\begin{proof}
Let $A$ be an $a\times b$-matrix representing $\phi$.  Let $C=\{1, \dots, b\}$ index the columns of $A$, and let $W=\{1, \dots, a\}$ index the rows of $A$.  If $I\subseteq C$ and $J\subseteq W$ then we write $A_{I,J}$ for the corresponding submatrix.

The Buchsbaum-Rim syzygy $\rho_{I'}$ is trivial if and only if all the $a\times a$ minors of $A_{I', W}$ are zero.  Let $a'=rank(A_{I',W})$ which by assumption is strictly less than $a$.  We may assume that the upper left $a'\times a'$ minor of $A_{I',W}$ is nonzero.  We set $I''=\{i_1, \dots, i_{a'+1}\}$ and $J''=\{1, \dots, a'\}$.  Let $\tau$ be the Buchsbaum-Rim syzygy of $A_{I'',J''}$.  Then $\tau\ne 0$ because $\det(A_{I''\setminus \{a'+1\}, J''})\ne 0$.  Also $(A_{I'',J''}) \cdot \tau =0$.  Thus:
\begin{equation*}
\begin{pmatrix} A_{I'',W} \end{pmatrix} \cdot \tau =\begin{pmatrix} A_{I'',J}\\  A_{I'', W-J''} \end{pmatrix}
\cdot \tau =\begin{pmatrix} 0 \\ * \end{pmatrix}
\end{equation*}
There exists an invertible matrix $B\in GL_{a}(k(x_1, \dots, x_n))$ such that:
$$B\cdot A_{I'',W}=\begin{pmatrix}A_{I'',J''} \\ 0 \end{pmatrix}$$
This gives:
$$0=(B\cdot A_{I'',W})\cdot \tau=B\cdot (A_{I'',W}\cdot \tau)$$
Since $B$ is invertible over $k(x_1, \dots, x_n)$ we conclude that $A_{I'',W}\cdot \tau=0$.  Thus $\tau$ is a syzygy on the columns of $A$ indexed by $I''$, and therefore $\tau$ represents a second syzygy of $M$.  The degree of $\tau$ is $\sum_{i\in I''} j_i$ which is strictly less than $\sum_{i\in I'} j_i$.
\end{proof}
We may now prove the second syzygy obstruction and the codimension obstruction.
\begin{proof}[Proof of the second syzygy obstruction in Proposition \ref{prop:BRobstructions}]
Apply Lemma \ref{lem:brcol1}, choosing $I'=\{1, \dots, a+1\}$.
\end{proof}

\begin{proof}[Proof of codimension obstruction in Proposition \ref{prop:BRobstructions}]
Recall that the module $M$ has minimal presentation:
$$\bigoplus_{\ell=1}^b S(-j_\ell)\rTo^{\phi} S^a\rTo M\rTo 0$$
Let $\Buchs(\phi)$ be the Buchsbaum-Rim complex of $\phi$.  Then we have
$$\codim(M)\leq \pdim(M) \leq \pdim(\Buchs(\phi))= b-a+1=\sum_{j} \beta_{1,j}(M) - a +1$$
Since $M$ has codimension $e$, we obtain the desired inequality.  In the case of equality, the maximal minors of $\phi$  contain a regular sequence of length $e$, so we may conclude:
$$\beta(M)=\beta (\Buchs(\phi))$$
\end{proof}

\begin{proof}[Proof of regularity obstruction in Proposition \ref{prop:BRobstructions}]
Since $M$ is Cohen-Macaulay of codimension $e$, we may assume by Artinian reduction that $M$ is finite length.  Recall that $b=\sum_j \beta_{1,j}(M)$ and let $\phi$ as in the proof of the codimension obstruction.  If $b=e+a-1$ then we have that
$$\reg(M)=\reg(\Buchs(\phi))=\sum_{\ell=1}^{b} j_\ell$$
We are left with the case that $b>e+a-1$.  Recall that $\sigma_1, \dots, \sigma_b$ is a basis of the syzygies of $M$.  We may change bases on the first syzygies by sending $\sigma_i \mapsto \sum p_{i\ell} \sigma_\ell$ where $\deg(p_{i\ell})=\deg \sigma_i-\deg \sigma_\ell=j_i-j_\ell$, and where the matrix $(p_{i\ell})$ is invertible over the polynomial ring.  We choose a generic $(p_{i\ell})$ which satisfies these conditions.  Let $\widetilde{\phi}$ be the map defined by $\sigma_b, \sigma_{b-1}, \dots, \sigma_{b-e-a+2}$.  Define $M':=\coker(\widetilde{\phi})$.  By construction, $M'$ has finite length, $\beta(M')=\beta(\Buchs(\widetilde{\phi}))$, and $M'$ surjects onto $M$.  Thus we have
$$\sum_{\ell=b-e-a+2}^f j_\ell=\reg(M')\geq \reg(M)=\overline{d_n}(M)$$
where the inequality follows from Corollary 20.19 of \cite{eis}.
\end{proof}

\begin{proof}[Proof of independence of obstructions in Proposition \ref{prop:BRobstructions}]
To show that the obstructions of Proposition \ref{prop:BRobstructions} are independent, we construct an explicit example of a virtual Betti diagram with precisely one of the obstructions.

For Proposition \ref{prop:BRobstructions} (\ref{prop:BR:2nd}) consider:
$$2\cdot \pi_{(0,1,5,6,7,8)}+\pi_{(0,5,6,7,8,9)}=\begin{pmatrix}
3&4&-&-&-&-\\ -&-&-&-&-&-\\ -&-&-&-&-&-\\
-&70&252&336&200&45
\end{pmatrix}$$
Then $\underline{d_2}=5>4$ so this diagram has a Buchsbaum-Rim second syzygy obstruction.

For Proposition \ref{prop:BRobstructions} (\ref{prop:BR:codim}) consider:
$$\pi_{(0,1,3,4)}=\begin{pmatrix} 1&2&-&- \\ -&-&2&1\end{pmatrix}$$
In this case $\sum \beta_{1,j}(\pi_{(0,1,3,4)})=2<3+1-1=3$.  More generally, the pure diagram $\pi_{(0,1,\alpha,\alpha+1)}$ has a codimension obstruction for any $\alpha\geq 3$.

For the case of equality in Proposition \ref{prop:BRobstructions} (\ref{prop:BR:codim}), consider:
$$\pi_{(0,1,6,10)}=\begin{pmatrix}
6&8&-&-\\
-&-&-&-\\
-&-&-&-\\
-&-&3&-\\
-&-&-&-\\
-&-&-&-\\
-&-&-&-\\
-&-&-&1\\
\end{pmatrix}$$
Since we have $\sum \beta_{1,j}(\pi_{(0,1,6,10)})=8=3+6-1$, the diagram $\pi_{(0,1,6,10)}$ should equal the Betti table of the Buchsbaum-Rim complex on a map: $\phi: R(-1)^8\to R^6$.  This is not the case.

For Proposition \ref{prop:BRobstructions} (\ref{prop:BR:reg}) consider:
$$2\cdot \pi_{(0,1,4,9,10)}=\begin{pmatrix}
6&10&-&-&-\\
-&-&-&-&-\\
-&-&6&-&-\\
-&-&-&-&-\\
-&-&-&-&-\\
-&-&-&-&-\\
-&-&-&6&4
\end{pmatrix}$$
Here we have $\overline{d_4}=10>9=\sum_{j=1}^9 1$.

\end{proof}
\section{A Linear Strand obstruction in Projective Dimension 3}\label{sec:linstrand}
In this section we build obstructions based on one of Buchsbaum and Eisenbud's structure theorems about free resolutions in the special case of codimension $3$ (see \cite{buchs-eis-struc}.)  The motivation of this section is to explain why the following virtual Betti diagrams do not belong to $\Bmod$:
\begin{equation}\label{eqn:minor:diags}
D=\begin{pmatrix} 2&4&3&- \\ - & 3&4&2 \end{pmatrix}, D'=\begin{pmatrix} 3&6&4&- \\ - & 4&6&3 \end{pmatrix}, D''=\begin{pmatrix} 2&3&2&- \\ - & 5&7&3 \end{pmatrix}
\end{equation}
Note that these diagrams do not have any of the Buchsbaum-Rim obstructions.  In fact, there are virtual Betti diagrams similar to each of these which are Betti diagrams of modules.  For instance, all of the following variants of $D$ {\em are} Betti diagrams of modules:
$$\begin{pmatrix} 2&4&1&- \\ - & 1&4&2 \end{pmatrix},  \begin{pmatrix} 2&4&2&- \\ - & 2&4&2 \end{pmatrix},\begin{pmatrix} 2&4&3&1 \\ - & 3&5&2 \end{pmatrix},  \begin{pmatrix}4&8&6&- \\ - & 6&8&4 \end{pmatrix}$$
The problem with $D$ must therefore relate to the fact that it has too many linear second syzygies to {\em not} contain a Koszul summand.  Yet whatever obstruction exists for $D$ must disappear upon scaling from $D$ to $2\cdot D$.  Incidentally, the theory of matrix pencils could be used to show that $D$ and $D''$ are not Betti diagrams.  We do not approach this problem via matrix pencils because we seek an obstruction which does not depend on the fact that $\beta_{0,0}=2$.

Let $S=k[x,y,z]$ and let $M$ be a graded $S$-module $M$ of finite length.  Further, let $M$ be generated in degree $0$ and with regularity $1$, so that $$\beta(M)=\begin{pmatrix}a&b&c&d\\- & b'&c'&d'  \end{pmatrix}$$
Let $T_i$ be the maps along the top row of the resolution of $M$ so that we have a complex:
$$0\rTo S(-3)^d \rTo{(T_3)} S(-2)^c\rTo^{(T_2)} S(-1)^b \rTo^{(T_1)} S^a \rTo 0$$
Similarly, let $U_j$ stand for matrices which give the maps along the bottom row of the resolution of $M$.  Observe that each $T_i$ and $U_j$ consists entirely of linear forms, and that $U_1=0$.   If $d\ne 0$, then the minimal resolution of $M$ contains a copy of the Koszul complex as a free summand. Since we may split off this summand, we assume that $d=0$.

We then have the following obstruction:

\begin{prop}[Maximal minor, codimension $3$ obstruction]\label{prop:maxminobs}
Let $M$ as defined above, and continue with the same notation.  Then:
$$b'-a+\rank(T_1)+\rank(U_3)\leq c'$$
Equivalently $c-d'+\rank(T_1)+\rank(U_3)\leq b$.
\end{prop}
\begin{proof}
By assumption, $M$ has a minimal free resolution given by:
$$0 \rTo S(-4)^{d'} \rTo^{\left( \begin{smallmatrix}Q_3 \\ U_3  \end{smallmatrix} \right)} S(-2)^c\oplus S(-3)^{c'} \rTo^{\left( \begin{smallmatrix} T_2 & Q_2 \\ 0 & U_2  \end{smallmatrix} \right)} S(-1)^b\oplus S(-2)^{b'} \rTo^{\left( \begin{smallmatrix} T_1 &Q_1  \end{smallmatrix} \right)} S^a \rTo M$$
Each $Q_i$ stands for a matrix of degree $2$ polynomials.  By \cite{buchs-eis-struc} we know that each maximal minor of the middle matrix is the product of a corresponding maximal minor from the first matrix and a corresponding maximal minor from the third matrix.

Let $\tau=\rank(T_1)$ and $\mu=\rank(U_3)$.  Since $\codim(M)\ne 0$, the rank of the matrix $\begin{pmatrix} T_1 & Q_1 \end{pmatrix}$ equals $a$.  By thinking of this matrix over the quotient field $k(x,y,z)$, we may choose a basis of the column space which contains $\tau$ columns from $T_1$ and $a-\tau$ columns from $Q_1$.  Let $\Delta_1$ be the determinant of the resulting $a\times a$ submatrix, and observe that $\Delta_1$ is nonzero.  Similarly, we may construct  a $d'\times d'$ minor $\Delta_3$ from the last matrix such that $\Delta_3$ is nonzero and involves $\mu$ rows from $U_3$ and $d'-\mu$ rows from $Q_3$.

Now consider the middle matrix:
	\begin{equation*}
		\bordermatrix{
		& c & c' \cr
		b & T_2&Q_2 \cr
		b'& 0 & U_2
		}
	\end{equation*}
Note that the columns of this matrix are indexed by the rows of the third matrix, and the rows of this matrix are indexed by the columns of the first matrix.  Choose the unique maximal submatrix such that the columns repeat none of the choices from $\Delta_3$ and such that the rows repeat none of the choices from $\Delta_1$.  We obtain a matrix of the following shape:
	\begin{equation*}
		\bordermatrix{
		& c-d'+\mu & c'-\mu \cr
		b-\tau & *&* \cr
		b'-a+\tau & 0 & *
		}
	\end{equation*}
Since $M$ has finite length, the Herzog-K\"uhl conditions in \cite{herz-kuhl} imply that $c'+c-d'=b+b'-a$, and thus this is a square matrix.  If $\Delta_2$ is the determinant of the matrix constructed above, then $\Delta_2=\Delta_1\Delta_3$ by \cite{buchs-eis-struc}.  Since $\Delta_1\ne0$ and $\Delta_3\ne 0$, this implies that the $(b'-a+\tau \times c-d'+\mu)$ block of zeroes in the lower left corner cannot be too large.  In particular,
$$b'-a+\tau + c-d'+\mu \leq b'+b-a$$
By applying the Herzog-K\"uhl equality $c'+c-d'=b+b'-a$, we obtain the desired results.
\end{proof}

We now prove a couple of lemmas which will allow us to use this obstruction to rule out the virtual Betti diagrams from line (\ref{eqn:minor:diags}).  We continue with the same notation, but without the assumption that $d=0$.

\begin{defn}
A matrix $T$ is \defi{decomposable} if there exists a change of coordinates on the source and target of $T$ such that $T$ becomes block diagonal or such that $T$ contains a column or row of all zeroes.  If $T$ is not decomposable then we say that $T$ is \defi{indecomposable}.
\end{defn}

\begin{lemma}\label{lem:indecomp}
If the Betti diagram $\begin{pmatrix}a&b&c&d\\- & b'&c'&d'  \end{pmatrix}$ is Cohen-Macaulay and is a minimal generator of $\Bmod$, then $T_1$ is indecomposable or $b=0$.
\end{lemma}
\begin{proof}
If we project the semigroup $\Bmod$ onto its linear strand via:
$$\begin{pmatrix}a&b&c&d\\- & b'&c'&d'  \end{pmatrix}\mapsto \begin{pmatrix}a&b&c&d \end{pmatrix}$$
then the image equals the semigroup of linear strands in $\Bmod$. By the Herzog-K\"uhl equations, the linear strand $\begin{pmatrix} a&b&c&d \end{pmatrix}$ of such a Cohen-Macaulay module determines the entire Betti diagram.  Hence the projection induces an isomorphism between the subsemigroup of Cohen-Macaulay modules of codimension 3 in $\Bmod$ and the semigroup of linear strands in $\Bmod$.  The modules with $T_1$ decomposable and $b\ne 0$ cannot be minimal generators of the semigroup of linear strands in $\Bmod$.
\end{proof}

\begin{lemma}\label{lem:kossumAND}
With notation as above we have:
\begin{enumerate}
     \item \label{lem:kossum} If there exists a free submodule $F\subseteq S(-1)^b$ such that  $F\cong S(-1)^3$ and such that the restricted map $T_1|_{F}$ has rank $1$, then the minimal resolution of $M$ contains a copy of the Koszul complex as a direct summand.
     \item \label{cor:rank2}  If $a=2, b\geq 3,$ and $T_1$ is indecomposable then $T_1$ has rank $2$.
\end{enumerate}
\end{lemma}
\begin{proof}
(\ref{lem:kossum})  Given the setup of the lemma, we have that $T_1|_{F}$ is an $a\times 3$ matrix of rank $1$ with linearly independent columns over $k$.  All matrices of linear forms of rank $1$ are compression spaces by \cite{eis-har-vecs}.  Since the columns of $T_1|_{F}$ are linearly independent, this means that we may choose bases such that:

\begin{equation}\label{eqn:koszulmat}
T_1|_{F}=\begin{pmatrix}x&y&z\\ 0&0&0\\ 0&0&0\\ \vdots & \vdots & \vdots \\  0&0&0 \end{pmatrix}
\end{equation}
The result follows immediately.

(\ref{cor:rank2})  Assume that $T_1$ has rank $1$ and apply part \eqref{lem:kossum} with $F$ any free submodule isomorphic to $S(-1)^3$.  We may then assume that the first three columns of $T_1$ look like (\ref{eqn:koszulmat}), and whether $b=3$ or $b>3$ it quickly follows that $T_1$ is decomposable.
\end{proof}

%


\begin{prop}\label{prop:dd'd''}
The virtual Betti diagrams
$$ D=\begin{pmatrix} 2&4&3&- \\ - & 3&4&2 \end{pmatrix}, D'=\begin{pmatrix} 3&6&4&- \\ - & 4&6&3 \end{pmatrix}, D''=\begin{pmatrix} 2&3&2&- \\ - & 5&7&3 \end{pmatrix}$$
do not belong to $\Bmod$.
\end{prop}

\begin{proof}
 Assuming $D$ were a Betti diagram, Lemma \ref{lem:indecomp} implies that the corresponding matrices $T_1$ and $U_3$ are indecomposable.  Lemma \ref{lem:kossumAND}  (\ref{cor:rank2}) implies that for $D$ as in (\ref{eqn:koszulmat}), we have $\rank T_1=\rank U_3=2$.  Observe that $D$ now has a maximal minor obstruction, as $c-d'+\tau+\mu=5$ while $b=4$.

Next we consider $D'$.  If $D'$ were a Betti diagram, then the corresponding $T_1$ and $U_3$ would both have to be indecomposable.  If also $T_1$ had rank $2$, then Theorem 1.1 of \cite{eis-har-vecs} would imply that it is a compression space.  In particular, $T_1$ would have one of the following forms:
$$
  \begin{pmatrix}
0&0&0&0&*&*\\
0&0&0&0&*&*\\
0&0&0&0&*&*
  \end{pmatrix},
  \begin{pmatrix}
0&0&0&0&0&*\\
0&0&0&0&0&*\\
*&*&*&*&*&*
  \end{pmatrix},
  \text{ or }
  \begin{pmatrix}
0&0&0&0&0&0\\
*&*&*&*&*&*\\
*&*&*&*&*&*
  \end{pmatrix}
$$
The matrix forms on the left and right fail to be indecomposable.  The middle form could not have linearly independent columns, since each $*$ stands for a linear form, and we are working over $k[x,y,z]$.
Thus $T_1$ and $U_3$ both have rank $3$, and it follows that $D'$ has a maximal minor obstruction.

In the case of $D''$, similar arguments show that the ranks of $T_1$ and $U_3$ must equal $2$ and $3$ respectively.  Thus $D''$ also has a maximal minor obstruction.
\end{proof}

\begin{example}\label{ex:nonconsec}
Note that the diagram $2\cdot D$ belongs to $\Bmod$.  In fact, if $N=k[x,y,z]/(x,y,z)^2$ and $N^\vee=Ext^3(N,S)$ then:
$$\beta(N\oplus N^\vee(4))=\begin{pmatrix} 1&-&-&- \\ - & 6&8&3 \end{pmatrix}+\begin{pmatrix} 3&8&6&- \\ - &-&-&1 \end{pmatrix}=\begin{pmatrix} 4&8&6&- \\ - & 6&8&4 \end{pmatrix}=2\cdot D$$
This diagram does not have a maximal minor obstruction as $\rank(T_1)=\rank(U_3)=3$.

Conversely, up to isomorphism the direct sum $N\oplus N^\vee(4)$ is the only module $M$ whose Betti diagram equals $2\cdot D$.  The key observation is that for $M$ to avoid having a maximal minor obstruction, we must have that $\rank (T_1)+\rank (U_3)\leq 6$. Thus we may assume that $M$ is determined by a $4\times 8$ matrix of linear forms which has rank $\leq 3$.  Such matrices are completely classified by \cite{eis-har-vecs} and an argument as in Proposition \ref{prop:dd'd''} can rule out all possibilities except that $M\cong N\oplus N^\vee(4)$.

In the proof of Theorem \ref{thm:negative} (\ref{thm:negative:nonconsec}), we will show that $3\cdot D$ does not belong to $\Bmod$.
\end{example}

\section{Special Cases when $\BN=\Bmod$}\label{sec:specialcases}

In this section we prove Proposition \ref{prop:specialcases} in two parts.  We first deal with projective dimension $1$.

\begin{prop}\label{prop:codim1}
Let $S=k[x]$ and fix $\underline{d}\leq \overline{d}$.  Then $\BN=\Bmod$.  The semigroup $\Bmod$ is minimally generated by pure diagrams.
\end{prop}
\begin{proof}
Let $D\in \BN$ be a virtual Betti diagram of projective dimension $1$.  We may assume that $D$ is a Cohen-Macaulay diagram of codimension $1$.  Then the Herzog-K\"uhl conditions \cite{herz-kuhl} imply that $D$ has the same number of generators and first syzygies.  List the degrees of the generators of $D$ in increasing order $\alpha_1\leq \alpha_2 \leq \dots \leq \alpha_s$, and list the degrees of the syzygies of $D$ in increasing order $\gamma_1\leq \gamma_2 \leq \dots \leq \gamma_s$.  Then $D\in \BN$ if and only if we have:
$$\alpha_{i}+1\leq \gamma_{i}$$
for $i=1, \dots, s$.  Choose $M$ to be a direct sum of the modules
$$M_i:=\coker (\phi_i: R(-\gamma_{i})\to R(-\alpha_i) )$$
where $\phi_i$ is represented by any element of degree $\gamma_{i}-\alpha_{i}$ in $R$.  Note that $\beta(M_i)$ equals the pure diagram $\pi_{(\alpha_i,\gamma_i)}$.  Thus $D\in \Bmod$ and $D=\beta(M)=\sum_i \pi_{(\alpha_i,\gamma_i)}$.
\end{proof}

Recall the definition of a level module \cite{boij}:
\begin{defn}
A graded module $M$ is a \defi{level module} if its generators are concentrated in a single degree and its socle is concentrated in a single degree.
\end{defn}

We now show that in the case of projective dimension $2$ level modules, the semigroups $\BN$ and $\Bmod$ are equal.
\begin{prop}\label{prop:codim2level}
Let $S=k[x,y]$ and fix $\underline{d}\leq \overline{d}$ such that $\underline{d_0}=\overline{d_0}$ and $\underline{d_2}=\overline{d_2}$.  Then $\BN=\Bmod$.
\end{prop}
\begin{proof}
We may assume that $\underline{d_0}=0$, and then we are considering the semigroup of level modules of projective dimension $2$ with socle degree $(\underline{d_2}-2)$.  Let $D\in \BN$ and let $c$ be a positive integer such that $cD\in \Bmod$.  Let $\vec{h}(D)=(h_0, h_1, \dots )$ be the Hilbert function of $D$.  The main result of \cite{sod-codim2} shows that $\vec{h}(D)$ is the Hilbert function of some level module of embedding dimension $2$ if and only if $h_{i-1}-2h_{i}+h_{i}\leq 0$ for all $i\leq \underline{d_2}-2$.

Since $cD\in \Bmod$, we know that $\vec{h}(cD)=c\vec{h}(D)$ is the Hilbert function of a level module.  Thus we have:
$$ch_{i-1}-2ch_{i}+ch_{i}\leq 0$$
The same holds when we divide by $c$, and thus $\vec{h}(D)$ is the Hilbert function of some level module $M$.  Since $M$ is also a level module, its Betti diagram must equal $D$.
\end{proof}

\begin{rmk}\label{rmk:unobs}
We conjectured above that $\BN=\Bmod$ in general in projective dimension $2$.  Some evidence for this conjecture is provided by computations of the author \cite{erman-thesis} which prove that all virtual Betti diagrams of projective dimension $2$ and generated in a single degree are ``unobstructed'' in the sense of Proposition \ref{prop:BRobstructions}.
\end{rmk}

\section{The Structure of $\BN \setminus \Bmod$}\label{sec:structure}
We are now prepared to prove Theorem \ref{thm:negative} and thus show that for projective dimension greater than $2$, the semigroups $\BN$ and $\Bmod$ diverge.  Recall the statement of Theorem 1.6:

\

\noindent {\bf Theorem 1.6:}  {\em Each of the following occurs in the semigroup of Betti diagrams:}
\begin{enumerate}
    \item   $\Bmod$ {\em is not necessarily a saturated semigroup.}
    \item   {\em The set $|\BN\setminus \Bmod|$ is not necessarily finite.}
    \item   {\em There exist rays of $\BN$ which are missing at least $(\dim S-2)$ consecutive lattice points.}
    \item   {\em There exist rays of $\BN$ where the points of $\Bmod$ are nonconsecutive lattice points.}
\end{enumerate}

\

The various pieces of the theorem follow from a collection of obstructed virtual Betti diagrams.
\begin{proof}[Proof of Part (1) of Theorem \ref{thm:negative}]
We will show that on the ray corresponding to
$$D_1=\favorite$$
every lattice point except $D_1$ itself belongs to $\Bmod$.  We have seen in (\ref{eqn:oldfavorite}) that $D_1\notin \Bmod$.  Certainly $2\cdot D_1\in \Bmod$ as $2\cdot D$ is the Buchsbaum-Rim complex on a generic $2\times 4$ matrix of linear forms.  We claim that $3\cdot D_1$ also belongs to $\Bmod$.  In fact, if we set $S=k[x,y,z]$ and:
$$M:=coker\begin{pmatrix}  x&y&z&0&0&0\\ 0&0&x&y&z&0\\ x+y & 0&0&x&y&z \end{pmatrix}$$
then the Betti diagram of $M$ is $3\cdot D_1$.\end{proof}

\begin{proof}[Proof of Part (2) of Theorem \ref{thm:negative}]
We will show that for all $\alpha\in \mathbb N$, the virtual Betti diagram:
 $$E_\alpha:= \begin{pmatrix}2+\alpha&3&2&- \\ -&5+6\alpha& 7+8\alpha&3+3\alpha\end{pmatrix}$$ does not belong to $\Bmod$.

Note that $E_0\notin \Bmod$ by Proposition \ref{prop:dd'd''}.  Imagine now that $\beta(M)=E_\alpha$ for some $\alpha$.  Let $T_1$ be the linear part of the presentation matrix of $M$ so that $T_1$ is an $(\alpha+2)\times 3$ matrix of linear forms.  Let $T_2$ be the $(3\times 2)$ matrix of linear second syzygies and write:
$$T_1\cdot T_2=\begin{pmatrix} l_{1,1} & l_{1,2} & l_{1,3} \\  l_{2,1} & l_{2,2} & l_{2,3} \\ \vdots & \vdots & \vdots
 \end{pmatrix}  \cdot \begin{pmatrix}s_{1,1}& s_{1,2}\\ s_{2,1}& s_{2,2}\\ s_{3,1}& s_{3,2}\end{pmatrix}$$
By Lemma \ref{lem:kossumAND} (\ref{lem:kossum}), the rank of $T_1$ must be at least $2$.  Let $T_1'$ be the top two rows of $T_1$, and by shuffling the rows of $T_1$, we may assume that the rank of $T_1'$ equals $2$.  So then  may assume that $l_{1,1}$ and $l_{2,2}$ are nonzero.  Since each column of $T_2$ has at least $2$ nonzero entries, it follows that the syzygies represented by $T_2$ remain nontrivial syzygies on the columns of $T_1'$.

It is possible however that columns of $T_1'$ are not $k$-linearly independent.  But since the rank of $T_1'$ equals $2$, we know that at least two of the columns are linearly independent.  Let $C$ be the cokernel of $T_1'$, and let $M':=C_{\leq 1}$ be the truncation of $C$ in degrees greater than $1$.  Then we would have:
$$\beta(M')=\begin{pmatrix}2&3&2&- \\ - & 5&7&3 \end{pmatrix} \text{ or } = \begin{pmatrix}2&2&2&-\\ -&*&*&*  \end{pmatrix}$$
The first case is impossible by Example \ref{prop:dd'd''}, and the second case does not even belong to $\BN$.
\end{proof}

\begin{proof}[Proof of Part (3) of Theorem \ref{thm:negative}]
Fix some prime $P\geq 2$ and let $S=k[x_1, \dots, x_{P+1}]$.  Consider the degree sequence:
$$d=(0,1,P+1,P+2,...,2P)$$
We will show that the first $P-1$ lattice points of the ray $r_d$ have a codimension obstruction.

Let $\overline{\pi}_d$ be the pure diagram of type $d$ where we fix $\beta_{0,0}(\overline{\pi}_d)=1$.  We claim that:
\begin{itemize}
   \item  $\beta_{1,1}(\overline{\pi}_d)=2$
   \item  All the entries of $\beta(\overline{\pi}_d)$ are positive integers.
\end{itemize}
For both claims we use the formula $\beta_{i,d_i}(\overline{\pi}_d)=\Pi_{k\ne i} \frac{d_k}{(-1)^k(d_i-d_k)}$.  We first compute:
$$\beta_{1,1}(\overline{\pi}_d)=\frac{(P+1)\cdot \dots \cdot (2P-1)\cdot (2P)}{(P\cdot(P+1) \dots (2P-1))}=\frac{2P}{P}=2$$
For the other entries of $\overline{\pi}_d$ we compute:
$$\beta_{i,d_i}(\overline{\pi}_d)=\frac{2P\cdot (2P-1)\cdot \dots \cdot (P+1)}{(i-2)!(P-i+1)!}\cdot \frac{1}{P+i-1}\cdot \frac{1}{P+i-2}=\frac{1}{P}\binom{P+i-3}{i-2} \binom{2P}{P-i+1}$$
Note that $\binom{2P}{P-i+1}$ is divisible by $P$ for all $i\geq 2$ and thus $\beta_{i,d_i}(\overline{\pi}_d)$ is an integer as claimed.

Since $\beta_{0,0}=1$ and $\beta_{1,1}=2$, the diagram $c\cdot \overline{\pi}_d$ hs a codimension obstruction for $c=1, \dots, P-1$.  Thus the first $P-1$ lattice points of the ray of $\pi_d$ do not correspond to Betti diagrams.
\end{proof}

\begin{proof}[Proof of Part (4) of Theorem \ref{thm:negative}]Consider the ray corresponding to $$D_2=\begin{pmatrix}2&4&3&-\\-&3&4&2  \end{pmatrix}$$ Proposition \ref{prop:dd'd''} shows that $D_2$ does not belong to $\Bmod$.  In Example \ref{ex:nonconsec} we showed that $2\cdot D_2$ does belong to $\Bmod$.  Thus, it will be sufficient to show that
$$3\cdot D_2= \begin{pmatrix}6&12&9&-\\ -&9&12&6  \end{pmatrix}$$
does not belong to $\Bmod$.

We assume for contradiction that there exists $M$ such that $\beta(M)=3 \cdot D_2$.  Then the minimal free resolution of $M$ is as below:
\begin{equation}\label{eqn:6129}
0 \rTo R(-4)^{6} \rTo^{\left( \begin{smallmatrix}Q_3 \\ U_3  \end{smallmatrix} \right)} R(-2)^9\oplus R(-3)^{12} \rTo^{\left( \begin{smallmatrix} T_2 & Q_2 \\ 0 & U_2  \end{smallmatrix} \right)} R(-1)^{12}\oplus R(-2)^9 \rTo^{\left( \begin{smallmatrix} T_1 &Q_1  \end{smallmatrix} \right)} R^6
\end{equation}
where $T_1, T_2, U_2$ and $U_3$ are matrices of linear forms.  By Proposition \ref{prop:maxminobs} we have that  $\rank(T_1)+\rank(U_3)\leq 9$.  Since the diagram $3\cdot D_2$ is Cohen-Macaulay and symmetric, we may use Remark \ref{rmk:duality} to assume that $\rank(T_1)\leq 4$.

We next use the fact that, after a change of coordinates, $T_2$ contains a second syzygy which involves only $2$ of the variables of $S$.  This fact is proven in Lemma \ref{lem:secondsyz} below.  Change coordinates so that the first column of $T_2$ represents this second syzygy and equals:
$$\begin{pmatrix} y\\ -x\\ 0 \\ \vdots \\ 0\end{pmatrix}$$
Since $T_1$ must be indecomposable, we may put $T_1$ into the form:
\begin{equation}\label{eqn:T1form}
T_1=\begin{pmatrix}
x&y&z&0& \dots & 0\\
0&0&*&*& \dots & *\\
\vdots&&&& &\vdots\\
0&0&*&*& \dots & *
\end{pmatrix}
\end{equation}
Now set $\widetilde{T_1}$ to be the lower right corner of $*$'s in $T_1$.  Since $\rank(T_1)\leq 4$ we have that $\rank(\widetilde{T}_1)\leq 3$.  Matrices of rank $\leq 3$ are fully classified, and by applying Corollary 1.4 of \cite{eis-har-vecs} we conclude that $\widetilde{T}_1$ is a compression space.  We can rule out the compression spaces cases where $\widetilde{T}_1$ has a column or a row equal to zero, or else $T_1$ would have been decomposable.  Thus $\widetilde{T}_1$ is equivalent to one of the two following forms:
$$
\begin{pmatrix}
0&0&0&0&0&0&0&0&0&*\\
0&0&0&0&0&0&0&0&0&*\\
0&0&0&0&0&0&0&0&0&*\\
0&0&0&0&0&0&0&0&0&*\\
*&*&*&*&*&*&*&*&*&*\\
*&*&*&*&*&*&*&*&*&*
\end{pmatrix} \text{ or }
\begin{pmatrix}
0&0&0&0&0&0&0&0&*&*\\
0&0&0&0&0&0&0&0&*&*\\
0&0&0&0&0&0&0&0&*&*\\
0&0&0&0&0&0&0&0&*&*\\
0&0&0&0&0&0&0&0&*&*\\
*&*&*&*&*&*&*&*&*&*
\end{pmatrix}
$$
If we subsitute the matrix on the left into the form for $T_1$ from \ref{eqn:T1form}, then we see that $T_1$ would have $8$ $k$-linearly indepdendent columns which are supported on only the bottom two rows.  Since all entries of $T_1$ are linear forms in $k[x,y,z]$, this is impossible.  We can similarly rule out the possibility of the matrix on the right.
\end{proof}

\begin{lemma}\label{lem:secondsyz}
If there exists a minimal resolution as in Equation (\ref{eqn:6129}), then the matrix $T_2$ contains a second syzygy involving only $2$ variables of $S$.
\end{lemma}
\begin{proof}
Assume that this is not the case and quotient by the variable $z$.  Then the quotient matrices $\overline{T_1}$ and $\overline{T_2}$ still multiply to $0$.  It is possible that after quotienting, some of the columns of $T_1$ are dependent.  However this is not possible for $T_2$.  For if some combination went to $0$ after quotienting by $z$, then there would exist a column of $T_2$, i.e. a second syzygy of $M$, which involves only the variable $z$.  This is clearly impossible.  Thus the columns of $\overline{T_2}$ are linearly independent.

Nevertheless, we know that the columns of a $6\times 12$ matrix of linear forms over $k[x,y]$ can satisfy at most $6$ independent linear syzygies.  By changing coordinates we may arrange that $3$ of the columns of $\overline{T_2}$ are ``trivial'' syzygies on $\overline{T_1}$.  By a ``trivial'' syzygy, we mean a column of $\overline{T_2}$ where the nonzero entries of that columns multiply with zero entries of $\overline{T_1}$.  For an example of how a nontrivial syzygy over $k[x,y,z]$ can become trivial after quotienting by $z$, consider:
$$\begin{pmatrix}x&z&0\\y&0&z  \end{pmatrix} \begin{pmatrix}z\\-x\\-y\end{pmatrix} \to \begin{pmatrix}x&0&0\\y&0&0  \end{pmatrix} \begin{pmatrix}0\\-x\\-y\end{pmatrix}$$

Change coordinates so that the first $3$ columns of $\overline{T_2}$ represent the trivial syzygies and are in Kronecker normal form.  By assumption, each column of $\overline{T_2}$ involves both $x$ and $y$, so these first $3$ columns must consist of combinations of the following Kronecker blocks:
$$
B_1=\begin{pmatrix}x\\ y \end{pmatrix}, B_2=\begin{pmatrix}x&0\\ y&x\\ 0&y \end{pmatrix}, B_3=\begin{pmatrix}x&0&0\\ y&x&0\\ 0&y&x\\ 0&0&y \end{pmatrix}
$$
Since each nonzero entry in the trivial part of $\overline{T_2}$ must multiply with a $0$ from $\overline{T_1}$, this forces certain columns of $\overline{T_1}$ to equal $0$.  More precisely, the number of nonzero rows in the trivial part of $\overline{T_2}$ is a lower bound for the number of columns of $\overline{T_1}$ which are identically zero.  The block decomposition shows that the trivial part of $\overline{T_2}$ has at least $4$ nonzero rows, and thus $\overline{T_1}$ has at least $4$ columns which are identically zero.

But now the nonzero part of $\overline{T_1}$ is a $6\times 8$ matrix of linear forms, and this can satisfy at most $4$ linear syzygies.  This forces two {\em additional} columns of $\overline{T_2}$ to be trivial syzygies which in turn forces more columns of $\overline{T_1}$ to equal zero, and so on.

Working through this iterative process, we eventually conclude that $\overline{T_1}$ contains $8$ columns which are identically zero.  This means that $T_1$ must have contained $8$ columns which involved only $z$.  But since $T_1$ is a $6\times 12$ matrix of linear forms with linearly independent columns, this is impossible.
\end{proof}

\begin{remark}
Consider the diagram:
$$
D=\frac{a}{2}\pi_{(0,1,2,4)}+\frac{b}{2}\pi_{(0,2,3,4)}=
\begin{pmatrix}
\frac{3a+b}{2} & 4a & 3a& - \\
- & 3b & 4b & \frac{a+3b}{2}
\end{pmatrix}
$$
Clearly $D\in\BN$ if and only if $a+b$ is even.  By an argument analogous to that in the proof of Theorem \ref{thm:negative} part \eqref{thm:negative:infty}, one can show that $D\notin \Bmod$ if $a=1$ or $b=1$.

Recent unpublished work of \cite{eis-schrey2} uses this example to greatly strengthen parts (2) and (4) of Theorem \ref{thm:negative}.  They show that $D\notin \Bmod$ whenever $a$ is odd (or equivalently whenever $b$ is odd).  Furthermore, they show that if $M$ is any module such that:
$$
\beta(M)=a'\pi_{(0,1,2,4)}+b'\pi_{(0,2,3,4)}
$$
then the module $M$ splits into a direct sum of the pure pieces.  Namely, $M\cong M'\oplus M''$ where $\beta(M')=a'\pi_{(0,1,2,4)}$ and $\beta(M'')=b'\pi_{(0,2,3,4)}$.  Similar results are shown to hold in codimension greater than $3$.

Based on a generalization of the methods of \cite{eis-schrey2}, we have recently computed all generators for $\Bmod$ when $\underline{d}=(0,1,2,3)$ and $\overline{d}=(1,2,3,4)$.  This computation will appear in \cite{erman-thesis}.
\end{remark}

\section{Further Questions}\label{sec:questions}
An ambitious question is whether we can find a better description of $\Bmod$ or compile a complete list of obstructions.  Here are several more specific questions.  A further list of questions is compiled in \cite{questions}.
\begin{enumerate}

\item{\bf Bounds on $\Bmod$:}  Can we bound the number of generators of the semigroup of Betti diagrams?  Can we bound the size of a minimal generator of the semigroup of Betti diagrams?

\item {\bf The behavior of single rays:}
Given a degree sequence $d$, what is the minimal $c_d$ such that $c_d\pi_d$ is the Betti diagram of some module?  In many cases where computation is feasible, it is known that the examples produced by \cite{efw} and \cite{eis-schrey} do not represent the first element of $\Bmod$ on the ray.  In some other cases, it is known that $\pi_d$ itself does not belong to $\Bmod$ so that $c_d$ is greater than $1$.  Can we find better lower and upper bounds for the integer $c_d$?

\item {\bf Dependence on characteristic:}  Schreyer's conjecture that the semigroup of Betti diagrams depends on the characteristic of $k$ has recently been proven by Kunte in Corollary 2.4.10 of \cite{kunte}.  In particular, Kunte shows that the virtual Betti diagram:
$$\begin{pmatrix}1&-&-&-&-&-\\
-&10&16&-&-&-\\
-&-&-&16&10&-\\
-&-&-&-&-&1\\
 \end{pmatrix}$$
is not the Betti diagram of a finite length algebra when the characteristic of $k$ equals $2$.  It was previously known that this is a Betti diagram when the characteristic of $k$ equals $0$.  To what extent does $\Bmod$ depend on the characteristic?  Can we find obstructions which only live in specific characteristics?

\end{enumerate}

\bibliographystyle{alpha}

\end{document}